\documentclass[a4paper,11pt]{article}
\usepackage{latexsym}
\usepackage{amssymb}
\usepackage{enumerate}
\usepackage{lscape}
\usepackage[shortlabels]{enumitem}
\usepackage{xcolor}
\usepackage{commath}
\usepackage{comment}

\usepackage{mathtools} 

\usepackage{amsfonts}
\usepackage{amsmath,amsthm}
\usepackage{hyperref}
\usepackage{algorithm}
\usepackage{algorithmic}
\usepackage{relsize}
\usepackage[title]{appendix}
\usepackage{framed}
\usepackage{boites}
\usepackage[pdftex]{graphicx}
\newtheorem{mainthm}{Theorem}

\newenvironment{@abssec}[1]{%
    \if@twocolumn

      \section*{#1}%
    \else

      \vspace{.05in}\footnotesize
      \parindent .2in
 {\upshape\bfseries #1. }\ignorespaces
    \fi}

    {\if@twocolumn\else\par\vspace{.1in}\fi}

\newenvironment{keywords}{\begin{@abssec}{\keywordsname}}{\end{@abssec}}

\newenvironment{AMS}{\begin{@abssec}{\AMSname}}{\end{@abssec}}

\newcommand\keywordsname{Key words}
\newcommand\AMSname{AMS subject classifications}
\newcommand\AMname{AMS subject classification}
\newcommand\restr[2]{{
\left.\kern-\nulldelimiterspace 
#1 
\vphantom{|} 
\right|_{#2} 
}}
\newtheorem{theorem}{Theorem}[section]
\newtheorem{lemma}[theorem]{Lemma}
\newtheorem{corollary}[theorem]{Corollary}
\newtheorem{proposition}[theorem]{Proposition}
\newtheorem{remark}[theorem]{Remark}
\newtheorem{definition}[theorem]{Definition}
\newtheorem{problem}{Problem}

\newtheorem{conjecture}[theorem]{Conjecture}

\makeatletter
\def\moverlay{\mathpalette\mov@rlay}
\def\mov@rlay#1#2{\leavevmode\vtop{%
   \baselineskip\z@skip \lineskiplimit-\maxdimen
   \ialign{\hfil$\m@th#1##$\hfil\cr#2\crcr}}}
\newcommand{\charfusion}[3][\mathord]{
    #1{\ifx#1\mathop\vphantom{#2}\fi
        \mathpalette\mov@rlay{#2\cr#3}
      }
    \ifx#1\mathop\expandafter\displaylimits\fi}
\makeatother

\def\XXint#1#2#3{{\setbox0=\hbox{$#1{#2#3}{\int}$}
\vcenter{\hbox{$#2#3$}}\kern-.5\wd0}}

\newcommand{\re}{\mathop{\mathrm{Re}}}

\newcommand{\link}{\mathop{\circ\kern-.35em -}}

 \newcommand{\Pairing}[2]{
 \prescript{}{{H^{-1/2}}}{\!\left\langle #1, #2\right\rangle_{H^{1/2}}}
 } 

\newcommand{\pairing}[2]{
{}_{{H^{-1/2}}}\hspace*{-1mm}\left\langle #1, #2\right\rangle_{H^{1/2}}
}  

\newcommand{\al}{\alpha}

\newcommand{\om}{\omega}

\newcommand\setbld[2]{\left\{ #1 \;\middle |\; #2\right\}}

\newcommand{\cdottone}{{\boldsymbol{\cdot}}}

\newcommand{\identmatrix}
{{\mathbf{I}}}

\numberwithin{equation}{section}

\title{A discrete-time overdetermined problem for the heat equation
}

\author{Lorenzo Cavallina \, Andrea Pinamonti \
}

\date{}

\begin{document}

\maketitle

\begin{abstract}
In this paper, we consider a parabolic counterpart of Serrin's overdetermined problem, in which the overdetermined condition (constant flux condition) is imposed only on a discrete infinite set of time values. We show that, under suitable regularity assumptions on the domain, such a discrete-time overdetermined problem admits a solution if and only if the domain is a ball. Remarkably, depending on the temporal scale, the same overdetermined condition captures either geometric or spectral information, yet both mechanisms lead to the same rigidity conclusion. We study both the case in which the constant flux condition is imposed on the boundary and the case in which the constant flux condition is imposed on an interior surface. We remark that the methods employed in our analysis do not depend on the location of the overdetermined surface but only on whether the sequence of time instants accumulates away from zero. Finally, we will show how this problem generalizes to complete Riemannian manifolds.
\end{abstract}


\maketitle

\pagestyle{plain}
\thispagestyle{plain}

\section{Introduction}\label{section introduction}

Let $\Omega$ be a bounded domain of $\mathbb{R}^N$ ($N\ge 2$).
We consider the following initial/boundary-value problem for the heat equation:
\begin{equation}\label{heat}
  \begin{cases}
      \partial_t u = \Delta u \quad \text{in }\Omega\times (0,\infty),\\
      u=1\quad \text{on } \Omega\times\{0\},\\
      u=0\quad \text{on }\partial\Omega\times(0,\infty).
  \end{cases}  
\end{equation}

In what follows, it will be convenient to think of $u$ as a function of time with values in a function space $X(\Omega)$, instead of a real-valued function of several variables. In such cases, we will employ the notation $t\mapsto u(t)\in X(\Omega)$, where $X(\Omega)$ is a function space of real-valued functions defined on $\Omega$ (which may depend on the context) and $u(t)(x):=u(x,t)$ for $x\in \Omega$ and $t>0$. 

This paper addresses the problem of \emph{characterizing Euclidean balls by the behavior of the heat flux} $\partial_\nu u(t)$ of the solution to \eqref{heat}. This can be seen as a natural extension of Serrin's overdetermined problem to a parabolic setting. In his celebrated paper \cite{serrin1971}, J.~Serrin studied the following elliptic overdetermined problem:
\begin{equation}\label{serrin odp}
    -\Delta U = 1\quad \text{in }\Omega,\quad U=0\quad \text{on }\partial\Omega,\quad 
    \partial_\nu U = c\quad \text{on }\partial\Omega, 
\end{equation}
where $\partial_\nu$ denotes the exterior normal derivative, and $c$ is a real constant. He showed that if $\Omega$ is a bounded domain of class $C^2$, then problem \eqref{serrin odp} admits a solution $U\in C^2(\overline{\Omega})$ if and only if $\Omega$ is a ball (and the solution $U$ is radial). Serrin's original proof relies on a newly introduced method, now referred to as the \emph{method of moving planes}, which refines Aleksandrov's \emph{reflection principle} (see \cite{aleksandrov1958uniqueness}). 
In the following years, many mathematicians worked on improving Serrin's result in multiple directions: by providing alternative methods of proof (see \cite{weinberger1971remark} for a new strategy using a $P$-function and the maximum principle), by extending it to the manifold setting (see \cite{kumaresan_prajapat1998serrin} for positive results in the hyperbolic space, and in the sphere for domains included in a hemisphere; and \cite{fall&minled&weth2018serrin_on_the_sphere} for counterexamples to symmetry in the sphere under no additional assumption), by weakening the regularity assumptions on the domain (\cite{vogel1992symmetry_and_regularity} showed that $C^1$ regularity for the boundary is enough to obtain symmetry if the overdetermination is considered in some suitable generalized sense, \cite{prajapat1998serrin} extended Serrin's result to domains with a cusp, and finally  \cite{figalli&zhang2025_serrin_rough_domains} showed that Serrin's result holds true among bounded domains of finite perimeter that satisfy some density bounds), and by extending the result to nonlinear operators (see \cite{Garofalo, MR3955113,MR1776351} for an extension to the $p$-Laplacian and \cite{MR2366129,MR2232009,MR2545870} for the $\infty$-Laplacian; further developments also include equations driven by fully nonlinear non-divergence-form operators, such as $k$-Hessian equations \cite{MR2448319}.) For a more general overview and a complete list of references, we suggest the surveys \cite{MR3802818,MR3893584}.

The parabolic counterpart of Serrin's problem has also been studied from several perspectives. 
A first important contribution in this direction is \cite{MR1011556}, in which the authors considered symmetry for a class of possibly degenerate parabolic equations, thus providing an early parabolic analogue of the Aleksandrov--Serrin theory. For the heat equation, a more specific line of research was later developed in \cite{MR2231874} through the study of \emph{stationary isothermic surfaces}, that is, spatial level surfaces of the temperature which do not depend on time. The results proved in \cite{MR2231874}, together with subsequent extensions provided in \cite{MR2375699}, show that such overdetermined parabolic information has strong geometric consequences and, in the bounded case, characterizes balls. This point of view is also closely related to the \emph{constant flow property}, introduced in the Riemannian setting in \cite{BisSavo, savo2016heat} and further investigated for two-phase conductors in \cite{CavallinaMagnaniniSakaguchi2021constantflowproperty}, and later for general multi-phase conductors in \cite{cavallina_poggesi2025elliptic&parabolic}. 

In the pursuit of extending Serrin's problem to the parabolic setting, one might first consider the following na\"{\i}ve question:\\
\emph{``For what bounded domain $\Omega\subset\mathbb{R}^N$ does the solution $u$ of \eqref{heat} also satisfy the following overdetermined condition at all times $t>0$ for some function $b(\cdottone)$?"}
\begin{equation}\label{naive extension}
\partial_\nu u(t)=b(t)\quad \text{on }\partial\Omega.
\end{equation}
In other words, such a problem aims to characterize all bounded domains satisfying the ``\emph{constant flow property}" at the boundary (see \cite{savo2016heat}, where such terminology was first introduced in the context of the heat flow on Riemannian manifolds, and \cite{CavallinaMagnaniniSakaguchi2021constantflowproperty}, where the authors characterized the geometry of two-phase heat conductors with the constant flow property).
Although maybe not obvious at first glance, requiring the overdetermined condition \eqref{naive extension} to hold for all $t>0$ is a very restrictive assumption. Indeed, if $u$ satisfies both \eqref{heat} and \eqref{naive extension} for all $t>0$, then one can readily check that the function
\begin{equation*}
U(x):=\int_0^\infty u(x,t)\, dt    
\end{equation*}
satisfies Serrin's overdetermined problem \eqref{serrin odp} for $c:= \int_0^\infty b(t)\, dt$. In other words, if \eqref{naive extension} holds \emph{for all} $t>0$, then $\Omega$ must be a ball. Analogous reduction techniques used to turn a parabolic overdetermined problem into a more manageable elliptic one via time integrals or integral transforms in the time variable are presented in \cite{sakaguchi2016rendiconti, Sakaguchi_hyperplanes2020, sakaguchibessatsu, CavallinaMagnaniniSakaguchi2021constantflowproperty,  cavallina_poggesi2025elliptic&parabolic}.

The previous observation shows that prescribing the overdetermined condition \eqref{naive extension} for every $t>0$ is too restrictive if one wishes to obtain genuinely parabolic phenomena, since the problem reduces to the classical elliptic setting. It is therefore natural to ask whether a weaker requirement in time still carries geometric information about the domain.

In what follows, we consider the following overdetermined boundary condition at infinitely many times.

\begin{equation}\label{boundary odc}
\begin{aligned}
\text{There exist sequences $(t_n)_{n}\subset (0,\infty)$ and $(b_n)_{n}\subset\mathbb{R}$ such that}\\
\partial_\nu u(t_n) \equiv b_n \quad \text{on }\partial\Omega \quad \text{for all }n\in\mathbb{N},
\end{aligned}
\end{equation}
where the equality is understood in the sense of \emph{weak conormal derivatives} (see Definition~\ref{def conormal}). 

At first glance, imposing the condition only at a sequence of times, as in \eqref{boundary odc}, might appear as a purely technical relaxation. However, this is not the case. Indeed, unlike condition \eqref{naive extension}, the discrete-time overdetermination does not allow for a direct reduction to an elliptic problem via time-integration. As a consequence, the problem retains an intrinsically parabolic nature, and its analysis requires a different set of tools, combining spectral information and short-time asymptotics. In this sense, the newly introduced condition \eqref{boundary odc} should be regarded not merely as a weakening of \eqref{naive extension}, but rather as a genuinely different regime where the temporal structure of the heat flow plays a crucial role.

At first glance, imposing the condition only at a sequence of times, as in \eqref{boundary odc}, might appear as a purely technical relaxation. However, this is not the case. The novelty of our approach lies in the fact that the discrete-time overdetermination prevents a direct reduction to an elliptic problem via time-integration. Consequently, the problem remains genuinely parabolic and exhibits two distinct rigidity mechanisms depending on the accumulation point $t^\star \in [0, \infty]$ of the sequence $(t_n)_n$:
\begin{itemize}
    \item \textbf{A spectral rigidity mechanism} when $t^\star \in (0, \infty]$, where the overdetermination encodes the long-time behavior of the heat flow and the spectral properties of the Dirichlet Laplacian;
    \item \textbf{A short-time geometric rigidity mechanism} when $t^\star = 0$, governed by the local geometry of the boundary via heat-content asymptotics.
\end{itemize}
This dichotomy highlights an intrinsically parabolic feature: the same boundary condition, observed at different temporal scales, triggers two independent phenomena that both ultimately enforce spherical symmetry.

Clearly, if $\Omega$ is a ball, then the solution to \eqref{heat} satisfies the overdetermined condition \eqref{boundary odc} for all times $t>0$, thus in particular, for any sequence $(t_n)_n$. In fact, we will show that satisfying \eqref{boundary odc} is sufficient to characterize the ball among domains with a suitable regularity assumption.

We now present our main results, which show that both regimes $t^\star\in (0,\infty]$ and $t^\star=0$ yield the same rigidity conclusion.
\begin{mainthm}\label{mainthm I}
Let $\Omega$ be a bounded Lipschitz domain of $\mathbb{R}^N$ ($N\ge 2$). Suppose that the solution to \eqref{heat} satisfies the overdetermined condition \eqref{boundary odc} and the sequence $(t_n)_{n}$ has an accumulation point $t^\star\in (0,\infty]$. Then, $\Omega$ is a ball.
\end{mainthm}

\begin{mainthm}\label{mainthm II}
Let $\Omega$ be a bounded domain of $\mathbb{R}^N$ ($N\ge 2$) whose boundary is of class~$C^\infty$. Suppose that the solution to \eqref{heat} satisfies the overdetermined condition \eqref{boundary odc} and the sequence $(t_n)_{n}$ has $t^\star=0$ as an accumulation point. Then, $\Omega$ is a ball.
\end{mainthm}

\begin{remark}[On the difference in the regularity assumptions of Theorems \ref{mainthm I}-\ref{mainthm II}]
As previously stated, the two cases $t^\star\in (0,\infty]$ and $t^\star=0$ translate to essentially different information on the behavior of the solution $u$. Indeed, while case $(i)$ determines the long-time behavior of the heat flow, thus giving us information on the eigenfunctions of the Dirichlet Laplacian on $\Omega$ (see the proof of Theorem~\ref{mainthm I} given in section \ref{section pf mainthm I}), case $(ii)$ determines the short-time behavior of the heat flow, thus giving us information on the curvature of $\partial\Omega$ (see the proof of Theorem~\ref{mainthm II} given in section \ref{section pf mainthm II}).
This essential difference is reflected in the distinct regularity assumptions imposed in Theorems \ref{mainthm I}-\ref{mainthm II}. 
\end{remark}

In addition to Theorems \ref{mainthm I}-\ref{mainthm II}, an analogous symmetry result can be obtained by imposing the overdetermination on a surface inside $\Omega$. 
Similarly to condition \eqref{boundary odc}, we can consider the following overdetermined condition, imposed on the boundary of a \emph{subdomain} $\omega$ of $\Omega$ for countably infinitely many instants of time.  

\begin{equation}\label{interior odc}
\begin{aligned}
\text{There exist a Lipschitz subdomain $\emptyset\ne\omega\subset \overline{\omega}\subset \Omega$ and}\\ \text{sequences $(t_n)_{n}$, $(\tau_n)_n\subset (0,\infty)$,\quad  $(a_n)_n$,  $(b_n)_{n}\subset\mathbb{R}$ such that}\\
u(\tau_n)\equiv a_n,\quad \partial_\nu u(t_n) \equiv b_n \quad \text{on }\partial\omega \quad \text{for all }n\in\mathbb{N},
\end{aligned}
\end{equation}
where the equalities above have to be understood in the sense of traces (for $u(\tau_n)$) and in the sense of weak conormal derivatives (for $\partial_\nu u(t_n)$). 

\begin{mainthm}\label{mainthm III}
Let $\Omega$ be a bounded domain of $\mathbb{R}^N$ ($N\ge 2$) whose boundary $\partial\Omega$ is made entirely of regular points for the Dirichlet Laplacian (see \cite[Chapter 8]{GilbargTrudinger}). Suppose that the solution to \eqref{heat} satisfies the overdetermined condition \eqref{interior odc} for some sequences $(t_n)_{n}$ and $(\tau_n)_n$. Assume that $(t_n)_n$ and $(\tau_n)_n$ have accumulation points $t^\star, \tau^\star\in (0,\infty]$ respectively. 
Then, $\omega$ and $\Omega$ are concentric balls.
\end{mainthm}

\begin{remark}[Why the overdetermined surface in Theorem~\ref{mainthm III} must be the boundary of a subdomain]
In the setting of Theorem~\ref{mainthm III}, we assumed that the \emph{overdetermined surface} $\partial\omega$ is the boundary of a subdomain $\omega\subset\overline{\omega}\subset\Omega$. We remark that this is not a technical assumption. Indeed, just requiring condition \eqref{interior odc} to simply hold on a closed manifold $\Gamma\subset \Omega$ is not enough to conclude that $\partial\Omega$ and $\Gamma$ are concentric spheres. A simple counterexample is given by the following annular  configuration:
\begin{equation*}
    \Omega:=B_R\setminus\overline{B_r}, \quad \Gamma:= \partial B_\rho, \quad (0<r<\rho<R).
\end{equation*}
\end{remark}
\medskip

\noindent\textbf{Plan of the paper.} This paper is organized as follows. In section \ref{section preliminaries}, we present some preliminary results on weak conormal derivatives and the eigenvalue expansion of the solution to the heat equation. In section \ref{section pf mainthm I}, we show Theorem \ref{mainthm I} by turning the discrete-time overdetermination \eqref{boundary odc} into an overdetermined condition on the Dirichlet eigenfunctions of the Laplacian and then conclude by making use of Serrin's theorem on rough domains \cite{figalli&zhang2025_serrin_rough_domains}.
In section \ref{section pf mainthm II}, we show Theorem \ref{section pf mainthm II} by turning the discrete-time overdetermination \eqref{boundary odc} into a condition on the mean curvature of $\partial\Omega$, and then conclude by the classical Aleksandrov's soap bubble theorem \cite{aleksandrov1958uniqueness}. 
In section \ref{section pf mainthm III}, we give a proof of Theorem \ref{mainthm III} along the lines of Theorem \ref{mainthm I}.
Finally, in section \ref{section manifolds}, we discuss how the above results generalize to the case of complete Riemannian manifolds such as space forms.

\section{Preliminaries}\label{section preliminaries}

In this section, we present some preliminary technical results concerning \emph{weak conormal derivatives} and the \emph{eigenvalue expansion} of the solution $u$ of \eqref{heat}. 

\begin{definition}[Weak conormal derivative]\label{def conormal}
Let $\Omega$ be a bounded Lipschitz domain of $\mathbb{R}^N$ ($N\ge 2$). Let $\varphi\in H^1(\Omega)$ be such that $\Delta \varphi\in L^2(\Omega)$. Then, we can define its conormal derivative 
\begin{equation*}
   \partial_\nu \varphi\in H^{-1/2}(\partial\Omega):=\left(H^{1/2}(\partial\Omega) \right)^* 
\end{equation*}
via the following duality pairing:
\begin{equation}\label{pa_nu via pairing}
\pairing{\partial_\nu\varphi}{\psi}:= \int_\Omega \Delta\varphi \widetilde\psi + \int_\Omega \nabla \varphi\cdot\nabla \widetilde\psi.
\end{equation}
Here, $\widetilde\psi$ denotes the harmonic extension of $\psi$, that is, the unique function $\widetilde\psi\in H^1(\Omega)$ satisfying 
\begin{equation}\label{weak harm extension definition}
  \restr{\widetilde\psi}{\partial\Omega}=\psi \quad \text{and} \quad \int_\Omega \nabla\widetilde\psi \cdot \nabla w =0 \quad \text{for all }w\in H_0^1(\Omega). 
\end{equation}

Finally, for any given constant $b\in\mathbb{R}$, the notation $\partial_\nu \varphi\equiv b$ will be used as a shorthand for the following:
\begin{equation*}
    \pairing{\partial_\nu \varphi}{\psi} = b\int_{\partial\Omega} \psi \quad \text{for all }\psi\in H^{1/2}(\partial\Omega).
\end{equation*}
\end{definition}

\begin{remark}[On the regularity assumptions in Definition \ref{def conormal}]
In Definition \ref{def conormal} we consider a function $\varphi\in H^1(\Omega)$ with $\Delta\varphi\in L^2(\Omega)$. One might wonder whether such a $\varphi$ automatically belongs to $H^2(\Omega)$.
Unfortunately, in general, this is not the case. Indeed, peculiar Lipschitz domains $\Omega$ are known to exist where the solution to the Laplace equation with data in $L^2(\Omega)$ and Dirichlet $0$ boundary condition fails to belong to $H^2(\Omega)$ (see \cite[Theorem~A, 2]{jerison_kenig1995inhomogeneous}). As a result, the gradient $\nabla\varphi$ only belongs to $L^2(\Omega)$ in general, and so it does not necessarily admit a trace in the classical sense.
\end{remark}

\begin{remark}
Notice that the pairing \eqref{pa_nu via pairing} coincides with a formal application of integration by parts. Also, notice that the second term in the right-hand side of \eqref{pa_nu via pairing} vanishes if $\varphi$ has constant trace on $\partial\Omega$ (thus, in particular, when $\varphi\in H_0^1(\Omega)$) by \eqref{weak harm extension definition}   
\end{remark}

\begin{remark}
The pairing \eqref{pa_nu via pairing} actually defines a bounded linear mapping from $H^{1/2}(\partial\Omega)$ to $\mathbb{R}$, that is, $\partial_\nu \varphi\in H^{-1/2}(\partial\Omega)$, as claimed. Linearity is obvious, while boundedness is a consequence of the following estimate (see \cite[Lemma 4.3]{mclean2000strongly}):
\begin{equation}\label{mclean estimate}
    \norm{\partial_\nu \varphi}_{H^{-1/2}(\partial\Omega)}\le C_1 \norm{\varphi}_{H^1(\Omega)}+C_2\norm{\Delta\varphi}_{L^2(\Omega)},
\end{equation}
for some $C_1,C_2>0$ independent of $\varphi$.
\end{remark}

Functions with a constant weak conormal derivative can actually be characterized without explicitly mentioning the constant at play. As the following result shows, this can be done by making use of the subspace of zero-average functions:  
\begin{equation}\label{zero average test functions}
    H_*^{1/2}(\partial\Omega):=\setbld{\psi\in H^{1/2}(\partial\Omega)}{\int_{\partial\Omega}\psi=0}.
\end{equation}

\begin{lemma}\label{lem equivalent conormal}
Let $\Omega$ and $\varphi$ satisfy the assumptions of Definition \ref{def conormal}. Then, the following are equivalent:
\begin{enumerate}[$(a)$]
    \item $\partial_\nu \varphi\equiv b$ for some $b\in \mathbb{R}$ in the sense of Definition \ref{def conormal}.
    \item $\pairing{\partial_\nu \varphi}{\psi}=0$ holds true for all $\psi\in H_*^{1/2}(\partial\Omega)$.
\end{enumerate}
\end{lemma}
\begin{proof} 
The implication $(a)\implies (b)$ follows readily from Definition \ref{def conormal}. In what follows, we will show $(b)\implies (a)$. To this end, assume $(b)$ and take an arbitrary $\psi\in H^{1/2}(\partial\Omega)$. Set
\begin{equation*}
    \overline\psi:=\frac{1}{|\partial\Omega|}\int_{\partial\Omega}\psi, \quad \psi_*:=\psi-\overline\psi\in H_*^{1/2}(\partial\Omega). 
\end{equation*}
We have
\begin{equation*}
 \begin{aligned}
 \pairing{\partial_\nu \varphi}{\psi}= \underbrace{\pairing{\partial_\nu}{\psi_*}}_{=0}+\pairing{\partial_\nu \varphi}{\overline{\psi}}\\
 =\overline{\psi}\ \pairing{\partial_\nu\varphi}{1} 
 = \overline{\psi}\int_{\Omega} \Delta \varphi \cdot 1 = \frac{1}{|\partial\Omega|} \int_{\partial\Omega}\psi \, \int_\Omega \Delta \varphi,
 \end{aligned}   
\end{equation*}
which is nothing but $(a)$ with $b= \frac{1}{|\partial\Omega|}\int_\Omega \Delta \varphi$.
\end{proof}

Let $D(\Delta_D)$ denote the ``domain of the Dirichlet Laplacian", defined as
\begin{equation*}
    D(\Delta_D):=\setbld{v\in H_0^1(\Omega)}{\Delta v \in L^2(\Omega) \text{ (in the sense of distributions)}}.
\end{equation*}
We recall that, for a function $v\in H_0^1(\Omega)$, the expression ``$\Delta v \in L^2(\Omega)$ in the sense of distributions" means that there exists some $f\in L^2(\Omega)$ such that
\begin{equation}\label{sense of distributions}
    \int_\Omega \nabla v \cdot \nabla w = -\int_\Omega f w \quad \text{for all }w\in H_0^1(\Omega).
\end{equation}
Moreover, in such a case, we will write $\Delta v=f$. 
We remark that the space $D(\Delta_D)$ becomes a Banach space when endowed with the following graph norm: 
\begin{equation*}
\norm{v}_{D(\Delta_D)}:=\norm{v}_{H_0^1(\Omega)}+\norm{\Delta v}_{L^2(\Omega)}.
\end{equation*}

\begin{lemma}\label{lem weak conormal is weak neumann}
Let $\Phi\in D(\Delta_D)$ satisfy $-\Delta \Phi=1\in L^2(\Omega)$.
If, in addition, its weak conormal derivative $\partial_\nu\Phi$ satisfies $\pairing{\partial_\nu \Phi}{\psi}=0$ for all $\psi\in H_*^{1/2}(\partial\Omega)$, then there exists some $b\in \mathbb{R}$ such that $\Phi$ is also a weak solution to the following boundary value problem:  
\begin{equation*}
 -\Delta \Phi = 1 \quad \text{in }\Omega,\quad  \partial_\nu\Phi = b \quad \text{on }\partial\Omega,   
\end{equation*}
that is, $\Phi\in H_0^1(\Omega)$ satisfies
\begin{equation}\label{weak neu}
 \int_\Omega \nabla\Phi\cdot\nabla w = \int_\Omega w + b\int_{\partial\Omega}w \quad \text{for all }w\in H^1(\Omega).
\end{equation}
\end{lemma}
\begin{proof}
By Lemma~\ref{lem equivalent conormal}, there exists some $b\in \mathbb{R}$ such that $\partial_\nu \Phi\equiv b$ in the sense of weak conormal derivatives. Take an arbitrary $w\in H^1(\Omega)$. We can decompose it as
\begin{equation*}
    w=w_0+\widetilde w,
\end{equation*}
where $\widetilde w\in H^1(\Omega)$ is the harmonic extension of $\restr{w}{\partial\Omega}$. In other words, $\restr{\widetilde w}{\partial\Omega}=\restr{w}{\partial\Omega}$ and
\begin{equation}\label{weak harm ext}
   \int_\Omega \nabla \widetilde w \cdot \nabla v =0 \quad \text{for all } v\in H_0^1(\Omega). 
\end{equation}
Also, notice that $w_0\in H_0^1(\Omega)$ by construction.

Using $\Phi\in H_0^1(\Omega)$ as a test function in \eqref{weak harm ext} yields
\begin{equation}\label{some equation 1}
    \int_\Omega \nabla\widetilde w\cdot \nabla \Phi=0.
\end{equation}
On the other hand, integrating $-\Delta\Phi=1$ against $w_0$ yields
\begin{equation}\label{some equation 2}
    \int_\Omega w_0 = -\int_\Omega \Delta \Phi w_0 = \int_\Omega \nabla \Phi\cdot\nabla w_0.
\end{equation} 
Now, combining \eqref{some equation 1}, \eqref{some equation 2}, and the fact that $\partial_\nu \Phi\equiv b$ in the sense of conormal derivatives yields
\begin{equation*}
\begin{aligned}
 \int_\Omega \nabla\Phi\cdot\nabla w = \int_\Omega \nabla\Phi\cdot \nabla w_0 + \underbrace{\int_\Omega\nabla\Phi\cdot \nabla \widetilde w}_{=0} = \int_\Omega w_0 = \int_\Omega w - \int_\Omega \widetilde w \\
 = \int_\Omega w + \int_\Omega \Delta\Phi \widetilde w =
\int_\Omega w + \Pairing{\partial_\nu \Phi}{\restr{w}{\partial\Omega}}=
 \int_\Omega w + b \int_{\partial\Omega}w. 
\end{aligned}
\end{equation*}
Since $w\in H^1(\Omega)$ was arbitrary, this concludes the proof.
\end{proof}

\subsection{On the eigenfunction expansion of the solution}

Consider the following eigenvalue problem:
\begin{equation}\label{eq dirichlet eigenvalues}
\text{Find $\phi$ such that } \int_\Omega \nabla\phi\cdot\nabla w = \lambda \int_\Omega \phi w \quad \text{for all }w\in H_0^1(\Omega). 
\end{equation}
It is known that, without any additional smoothness assumptions on $\Omega$, problem \eqref{eq dirichlet eigenvalues} admits an increasing sequence of positive eigenvalues, which will be counted \emph{with multiplicity} as follows:
\begin{equation*}
    0<\lambda_1<\lambda_2\le \lambda_3\le \dots \le \lambda_k\le \lambda_{k+1}\le \dots \nearrow \infty.
\end{equation*}
Also, let $(\phi_k)_{k\in\mathbb{N}}\subset H_0^1(\Omega)$ denote an orthonormal family (in $L^2(\Omega)$) of eigenfunctions of \eqref{eq dirichlet eigenvalues}. 
We remark that sometimes it will be more convenient to index the eigenvalues of \eqref{eq dirichlet eigenvalues} by counting them \emph{without multiplicity}, as 
\begin{equation*}
0<\Lambda_1<\Lambda_2<\Lambda_3<\dots<\Lambda_k<\Lambda_{k+1}<\dots\nearrow\infty.
\end{equation*}
In this case, $\Lambda_k$ will be defined inductively by 
\begin{equation*}
    \Lambda_1:=\lambda_1, \quad \Lambda_{k+1}:=\min \setbld{\lambda_j}{j\in \mathbb{N},\quad \lambda_j>\Lambda_k}\quad \text{for }k\ge 1.
\end{equation*}
Also, for $k\in \mathbb{N}$, $V_k$ will denote the eigenspace corresponding to $\Lambda_k$:
\begin{equation*}
    V_k:=\setbld{\phi\in H_0^1(\Omega)}{-\Delta\phi=\Lambda_k\phi}.
\end{equation*}

\begin{lemma}\label{lem some convergence estimate}
Let $r\in \mathbb{R}$ and $\tau>0$. Then the following series converge:
\begin{equation}\label{important series}
    \sum_{k=1}^\infty \lambda_k^r e^{-\lambda_k \tau}, \quad    \sum_{k=1}^\infty \Lambda_k^r e^{-\Lambda_k \tau}.
\end{equation}
\end{lemma}
\begin{proof}
First, we recall the following well-known lower bound for $\lambda_k$ (see \cite[Corollary~1]{li_yau1983schrodinger&eigenvalue_problem}): 
\begin{equation}\label{li-yau}
 \lambda_k\ge C(N,|\Omega|) k^{2/N},  
\end{equation}
where $C(N,|\Omega|)$ is an explicit positive constant depending only on $N$ and $|\Omega|$.

Notice that, for $k$ large enough, we get
\begin{equation*}
 k^2\Lambda_k^r e^{-\Lambda_k \tau}\le k^2\lambda_k^r e^{-\lambda_k\tau}\le C(N,|\Omega|)^{-N} \lambda_k^{N+r}e^{-\lambda_k \tau},
\end{equation*}
where we have made use of \eqref{li-yau} in the last inequality. 
Now, since $\lambda_k\to\infty$ when $k\to\infty$, the above implies
\begin{equation*}
\lambda_k^r e^{-\lambda_k \tau} = o(1/k^2), \quad \Lambda_k^r e^{-\Lambda_k \tau} = o(1/k^2) \quad \text{as }k\to\infty, 
\end{equation*}
whence the series in \eqref{important series} converge, as claimed. 
\end{proof}

The following immediate corollary of Lemma~\ref{lem some convergence estimate} will turn out to be useful in studying overdetermined conditions in case $t^\star\in (0,\infty)$.
\begin{corollary}\label{cor holomorphic function}
Let $(\gamma_k)_k\subset\mathbb{C}$ be a sequence satisfying 
$|\gamma_k|\le D \Lambda_k$ for some positive constant $D$ independent of $k$. Then, the series 
\begin{equation*}
    f(z):=\sum_{k=1}^\infty \gamma_k e^{-\Lambda_k z}
\end{equation*}
defines a holomorphic function in $\setbld{z\in \mathbb{C}}{\re{z}>0}$. 
\end{corollary}
\begin{proof}
   Since each term in the series is a holomorphic function, it will suffice to show that the series converges uniformly on any compact subset $K\subset\setbld{z\in \mathbb{C}}{\re z>0}$.

   To this end, set $\tau := \min_{z\in K} \re z >0$ and notice that 
   \begin{equation*}
       \left|\gamma_k e^{-\Lambda_k z}\right|\le D \Lambda_k e^{-\Lambda_k \tau}, 
   \end{equation*}
whence, Lemma~\ref{lem some convergence estimate} implies the desired convergence. As a result, $f$ is holomorphic in $K$, being the uniform limit of a sequence of holomorphic functions. We conclude by the arbitrariness of $K\subset\setbld{z\in\mathbb{C}}{\re z>0}$.
\end{proof}

In what follows, we will show how to express the solution to \eqref{heat} as an infinite series through an eigenfunction expansion. We will focus on the convergence properties of the series expansion. 

For each $k\in \mathbb{N}$, let $\alpha_k:=\int_\Omega \phi_k$ and consider the following series:
\begin{equation}\label{eigenfunction expansion}
u(t):=\sum_{k=1}^\infty \alpha_k e^{-\lambda_k t}\phi_k. 
\end{equation}
Sometimes it will be more convenient to express \eqref{eigenfunction expansion} by gathering together the elements belonging to the same eigenspace, as follows:
\begin{equation}\label{eigenfunction expansion II}
 u(t):=\sum_{k=1}^\infty e^{-\Lambda_k t} \Phi_k,\quad \text{where }\Phi_k:=\sum_{\substack{j\in \mathbb{N}\\ \lambda_j=\Lambda_k}} \alpha_j\phi_j.
\end{equation}

\begin{proposition}\label{prop convergence of series}
The series \eqref{eigenfunction expansion}--\eqref{eigenfunction expansion II}
converge, to the same function, in the norms of $C^0\big([0,\infty), L^2(\Omega)\big)$ and $C^1\big([\tau,\infty), D(\Delta_D)\big)$ for all $\tau>0$.
\end{proposition}
\begin{proof}
First, we will show convergence in the $C^0\left([0,\infty), L^2(\Omega)\right)$-norm. 
To this end, we will show that both \eqref{eigenfunction expansion}--\eqref{eigenfunction expansion II} are Cauchy sequences. For natural numbers $m\ge n$, the orthonormality of the family $\{\phi_k\}_k$ yields the following 
\begin{equation*}
\begin{aligned}
    \norm{\sum_{k=n}^m \al_k e^{-\lambda_k t }\phi_k}_{C^0\left([0,\infty),L^2(\Omega)\right)}= 
    \max_{t\in [0,\infty)}\norm{\sum_{k=n}^m \alpha_k e^{-\lambda_k t} \phi_k}_{L^2(\Omega)}\\
    =\max_{t\in [0,\infty)}\left(\sum_{k=n}^m |\alpha_k|^2 e^{-2\lambda_k t} \norm{\phi_k}^2_{L^2(\Omega)} \right)^{1/2}  = \left( \sum_{n}^m |\alpha_k|^2\right)^{1/2}.  
\end{aligned}
\end{equation*}
Since the above converges to $0$ as $m,n\to\infty$ by the definition of $\alpha_k$, we conclude that \eqref{eigenfunction expansion} converges in the $C^0\left([0,\infty),L^2(\Omega)\right)$-norm as claimed. 

Now, to show convergence in the $C^1\left([\tau,\infty),D(\Delta_D)\right)$-norm, it will be enough to show that \eqref{eigenfunction expansion} is uniformly absolutely convergent in the same norm. We have
\begin{equation*}
\begin{aligned}
    \norm{\al_k e^{-\lambda_k t} \phi_k}_{C^1\left([\tau,\infty), D(\Delta_D)\right)}= \max_{t\in [\tau,\infty)}|\alpha_k| e^{-\lambda_k t} \norm{\phi_k}_{D(\Delta_D)}+
    \max_{t\in [\tau,\infty)}\lambda_k |\alpha_k| e^{-\lambda_k t} \norm{\phi_k}_{D(\Delta_D)}\\
    = (\lambda_k+1)|\alpha_k|e^{-\lambda_k \tau} \norm{\phi_k}_{D(\Delta_D)}
    =(\lambda_k+1)|\alpha_k|e^{-\lambda_k \tau} (\sqrt{\lambda_k}+\lambda_k) \le  
    (\sqrt{\lambda_k}+\lambda_k)(\lambda_k+1)|\Omega|^{1/2}e^{-\lambda_k \tau}.
\end{aligned}
\end{equation*}
The desired convergence then follows from Lemma~\ref{lem some convergence estimate}. 

Finally, since the sequence of partial sums of \eqref{eigenfunction expansion II} is nothing but a subsequence of those of \eqref{eigenfunction expansion}, \eqref{eigenfunction expansion II} converges to the same function as \eqref{eigenfunction expansion} in the norms of\\ $C^0\left([0,\infty),L^2(\Omega)\right)$ and $C^1\left([\tau,\infty), D(\Delta_D)\right)$ as well. This concludes the proof.
\end{proof}

\begin{corollary}\label{cor termwise whence u solves}
Let $u=u(t)$ be the function defined via the eigenfunction expansions \eqref{eigenfunction expansion}--\eqref{eigenfunction expansion II}. Then, the differential operators $\partial_t$, $\Delta$, and $\partial_\nu$ act on $u$ term-wise. In other words, for all $t>0$, we have:
  \begin{eqnarray}
\partial_t u(t)= \Delta u(t) = \sum_{k=1}^\infty -\lambda_k \alpha_k e^{-\lambda_k t}\phi_k= \sum_{k=1}^\infty -\Lambda_k e^{-\Lambda_k t} \Phi_k, \label{cor dt De}\\
\partial_\nu u(t) = \sum_{k=1}^\infty \alpha_k e^{-\lambda_k t} \partial_\nu \phi_k = \sum_{k=1}^\infty e^{-\Lambda_k t}\partial_\nu \Phi_k, \label{cor pa_nu}
  \end{eqnarray}
where convergence holds in $C^0\big([\tau,\infty), D(\Delta_D)\big) \cap C^1\big([\tau,\infty), L^2(\Omega)\big)$ for \eqref{cor dt De} and in\\
$C^1\big([\tau,\infty), H^{-1/2}(\partial\Omega)\big)$ for \eqref{cor pa_nu}
(both, for any given $\tau>0$). In particular, $u$ solves \eqref{heat}.
\end{corollary}
\begin{proof}
The claim follows by combining Proposition~\ref{prop convergence of series} with the continuity of the following operators:
\begin{equation*}
\begin{aligned}
\partial_t:\;& C^1\left([\tau,\infty), D(\Delta_D)\right) \longrightarrow C^0\left([\tau,\infty), D(\Delta_D)\right),\\  
\Delta:\;& C^1\left([\tau,\infty), D(\Delta_D)\right) \longrightarrow C^1\left([\tau,\infty), L^2(\Omega)\right),\\  
\partial_\nu: \;&C^1\left([\tau,\infty), D(\Delta_D)\right) \longrightarrow C^1\left([\tau,\infty), H^{-1/2}(\partial\Omega)\right).  
\end{aligned}
\end{equation*}
\end{proof}

\section{Proof of Theorem~\ref{mainthm I}}\label{section pf mainthm I}

In this section, we give a proof of Theorem~\ref{mainthm I}. 

\begin{proof}[Proof of Theorem~\ref{mainthm I}]
Let $u$ be the solution to \eqref{heat} and assume that $u$ satisfies \eqref{boundary odc} for a sequence $(t_n)_n$ with an accumulation point $t^\star\in (0,\infty]$. By the unique solvability of \eqref{heat}, Corollary~\ref{cor termwise whence u solves} implies that
\begin{equation*}
    u(t):=\sum_{k=1}^\infty e^{-\Lambda_k t} \Phi_k, 
\end{equation*}
where $\Phi_k$ are defined according to \eqref{eigenfunction expansion II}.  
Take now an arbitrary function $\psi\in H_*^{1/2}(\Omega)$. Following Definition \ref{def conormal}, overdetermined condition \eqref{boundary odc} then yields:
\begin{equation}\label{conormal derivative against psi with zero average}
0= \pairing{\partial_\nu u(t_n)}{\psi} = \sum_{k=1}^\infty \gamma_k e^{-\Lambda_k t_n},\quad \text{where } \gamma_k:= \pairing{\partial_\nu \Phi_k}{\psi}.
\end{equation}
Here we made use of \eqref{cor pa_nu} of Corollary~\ref{cor termwise whence u solves} to justify the interchange between the duality pairing and the series in the above. 

In what follows, we will give an upper bound for $|\gamma_k|$.  First, \eqref{mclean estimate} (combined with the Poincar\'e inequality) yields:
\begin{equation*}\label{estimate gamma_k}
\begin{aligned}
|\gamma_k|=\left|\pairing{\partial_\nu \Phi_k}{\psi}\right|\le \norm{\partial_\nu \Phi_k}_{H^{-1/2}(\partial\Omega)}\norm{\psi}_{H^{1/2}(\partial\Omega)}\\
\le \left\{ C_1'\norm{\Phi_k}_{H_0^1(\Omega)} + C_2\norm{\Delta\Phi_k}_{L^2(\Omega)}        \right\}\norm{\psi}_{H^{1/2}(\partial\Omega)}\\
\le\left( C_1' \Lambda_k^{1/2}+C_2\Lambda_k   \right) \norm{\Phi_k}_{L^2(\Omega)} \norm{\psi}_{H^{1/2}(\partial\Omega)}.
\end{aligned}
\end{equation*}
Since, $\Phi_k$ is nothing but the orthogonal projection of the constant function $1\in L^2(\Omega)$ onto the eigenspace $V_k$, the bound $\norm{\Phi_k}_{L^2(\Omega)}\le \norm{1}_{L^2(\Omega)}=|\Omega|^{1/2}$ holds for all $k$. Finally, since $\Lambda_k\to \infty$ as $k\to\infty$, there exists a constant $D>0$ independent of $k$ (but that might depend on $\psi$) such that 
\begin{equation}\label{bounds on gamma_k}
    |\gamma_k|\le D \Lambda_k \quad \text{for all }k\in \mathbb{N}.
\end{equation}

Now, we will deal with cases $t^\star=\infty$ and $t^\star\in (0,\infty)$ separately.

\noindent\textbf{Case $t^\star=\infty$}\\
Without loss of generality, assume that $(t_n)_n$ satisfies $\lim_{n\to\infty} t_n=\infty$.
Then \eqref{conormal derivative against psi with zero average} yields:
\begin{equation}\label{take out the first term}
    0= 0\cdot e^{\Lambda_1 t_n} = \gamma_1 + \sum_{k=2}^\infty \gamma_k e^{(\Lambda_1-\Lambda_k)t_n}.
\end{equation}
In what follows, we will take the limit of \eqref{take out the first term} as $n\to\infty$.

To this end, first notice that, since $\Lambda_1<\Lambda_k$ for all $k\ge 2$ and $n\in \mathbb{N}$, we have:
\begin{equation}\label{let's dominate}
    \left|\gamma_k e^{(\Lambda_1-\Lambda_k)t_n} \right|\le |\gamma_k|e^{(\Lambda_1-\Lambda_k)\tau}, \quad \text{where }\tau:=\min_{n\in\mathbb{N}}t_n>0.
\end{equation}

As a result, the estimate \eqref{bounds on gamma_k} combined with Lemma~\ref{lem some convergence estimate} readily implies that the series in \eqref{take out the first term} is dominated by a summable series uniformly in $n$. Thus, by the dominated convergence theorem (with respect to the counting measure on $\mathbb{N}$), we can take the limit of \eqref{take out the first term} as $n\to\infty$ to obtain 
\begin{equation}\label{gamma_1=0}
\gamma_1=0.    
\end{equation}

The same process can be repeated inductively to \eqref{conormal derivative against psi with zero average}, leading to 
\begin{equation}\label{gamma_k=0}
 \gamma_k:=   \pairing{\partial_\nu \Phi_k}{\psi}=0 \quad \text{for all }k\in \mathbb{N} \text{ and }\psi\in H_*^{1/2}(\partial\Omega).
\end{equation}
Consider now the following auxiliary function:
\begin{equation}\label{serrin solution diy}
    \Phi:=\sum_{k=1}^\infty \lambda_k^{-1}\alpha_k \phi_k=
    \sum_{k=1}^\infty \Lambda_k^{-1}\Phi_k.
\end{equation}
By construction, \eqref{serrin solution diy} converges in the $D(\Delta_D)$-norm. As a result, we have 
\begin{equation*}
\begin{aligned}
-\Delta \Phi &= \sum_{k=1}^\infty \Phi_k =1,\\  \pairing{\partial_\nu\Phi}{\psi}&=\sum_{k=1}^\infty \Lambda_k^{-1} \pairing{\partial_\nu \Phi_k}{\psi} = 0 \quad \text{for }\psi\in H_*^{1/2}(\partial\Omega).
\end{aligned}
\end{equation*}
Thus, by Lemma~\ref{lem weak conormal is weak neumann}, $\Phi$ is a weak solution to the following overdetermined problem: 
\begin{equation*}
    -\Delta\Phi=1\quad \text{in }\Omega,\quad \Phi=0\quad \text{on }\partial\Omega, \quad \partial_\nu \Phi=b\quad \text{on }\partial\Omega
\end{equation*}
for some constant $b\in \mathbb{R}$. By \cite{figalli&zhang2025_serrin_rough_domains}, $\Omega$ must be a ball, as claimed.

\noindent\textbf{Case $t^\star\in(0,\infty)$}\\
Without loss of generality, assume that $(t_n)_n$ converges to some $t^\star\in (0,\infty)$. Combining \eqref{bounds on gamma_k} and Corollary~\ref{cor holomorphic function} yields that the function
\begin{equation*}
   f(z):=\sum_{k=1}^\infty \gamma_k e^{-\Lambda_k z} 
\end{equation*}
is holomorphic in $\setbld{z\in \mathbb{C}}{\re z>0}$. On the other hand, by \eqref{conormal derivative against psi with zero average}, $(t_n)_n$ is a sequence of zeros for $f$ with $t^\star$ as an accumulation point. Thus, by the identity theorem, $f\equiv 0$ in the whole $\setbld{z\in\mathbb{C}}{\re z>0}$. In particular, $f(t)=0$ for all $t>0$. Since $\psi\in H_*^{1/2}(\partial\Omega)$ was arbitrary, this means that the function $u$ satisfies overdetermined condition \eqref{boundary odc} at all times $t>0$, thus in particular, along a sequence $(t_n)_n$ satisfying $\lim_{n\to \infty }t_n=\infty$. The problem is thus reduced to the case $t^\star=\infty$, completing the proof.
\end{proof}

\begin{remark}
Under stronger regularity assumptions, say, $\partial\Omega\in C^{2,\alpha}$ ($0<\alpha<1$), then a shorter proof is available. Indeed, directly from \eqref{gamma_1=0}, one concludes that $\Phi_1$ satisfies
\begin{equation*}
   -\Delta\Phi_1=\Lambda_1\Phi_1\quad \text{in }\Omega, \quad \Phi_1=0\quad \text{on } \partial\Omega, \quad \partial_\nu \Phi_1=b\quad \text{on }\partial\Omega 
\end{equation*}
for some constant $b\in\mathbb{R}$. Moreover, since $\Phi_1>0$ in $\Omega$ by Krein--Rutman's theorem, the desired conclusion readily follows from \cite[Theorem 2]{serrin1971}.

We remark that, unfortunately, the shortcut above is not available in our setting, as, to the best of our knowledge, no analogue of \cite[Theorem 2]{serrin1971} is (yet) known to hold for Lipschitz domains. 
\end{remark}
\newpage
\section{Proof of Theorem~\ref{mainthm II}}\label{section pf mainthm II}
In this section, we give a proof of Theorem~\ref{mainthm II}.
\begin{proof}[Proof of Theorem~\ref{mainthm II}]
By \cite[Theorem~1.2 and Theorem~1.3 (1): (a),(b)]{vandenberg&Gilkey&Seeley2008heat_content_asymptotics} or \cite[Theorem 2.3.3: 1., 2., 3.]{gilkey2003asymptotic}, for every $\psi\in C^\infty(\Omega)$ with $\Delta\psi=0$ we have
\begin{equation}\label{eq:star}
f(t):=\int_{\Omega} u(t^2)\,\psi
=\int_{\Omega}\psi
-\frac{2t}{\sqrt{\pi}}\int_{\partial\Omega}\psi
+\frac{t^2}{2}\int_{\partial\Omega} H\,\psi
+o(t^2)
\qquad \text{as }t\to 0^+,
\end{equation}
where $H$ is the mean curvature of $\partial\Omega$, that is, the sum of the principal curvatures of $\partial\Omega$ defined with respect to the inward unit normal. In particular, if $\partial\Omega$ is a sphere of radius $R$, then $H=\frac{N-1}{R}$ (see \cite[page 6] {gilkey2003asymptotic} or \cite{vandenberg&Gilkey&Seeley2008heat_content_asymptotics}).  

On the other hand, for $t>0$,
\begin{align*}
f'(t)
&=2t\int_{\Omega} u_t(t^2)\,\psi
=2t\int_{\Omega} \Delta u(t^2)\,\psi \\
&=2t\int_{\partial\Omega} \partial_{\nu} u(t^2)\,\psi,
\end{align*}
where the last equality follows by integration by parts using $\Delta \psi=0$.

That is, if, in addition
\begin{equation*}
    \int_{\partial\Omega}\psi=0
\quad\text{and}\quad
t=\sqrt{t_n},
\end{equation*}
we get
\begin{equation*}
f'(\sqrt{t_n})
=2\sqrt{t_n}\,\Pairing{ \partial_{\nu} u(t_n)}{\psi}
=0
\qquad \forall n\in\mathbb{N},
\end{equation*}

and hence
\begin{equation*}
f''(0)=0.
\end{equation*}

On the other hand, differentiating \eqref{eq:star} gives
\begin{equation*}
f''(0)=\int_{\partial\Omega} H\,\psi.  
\end{equation*}

Since $\psi$ is arbitrary (subject to $\Delta\psi=0$ and $\int_{\partial\Omega}\psi=0$), we conclude that $H$ is constant on $\partial\Omega$, whence $\Omega$ is a ball by Aleksandrov's soap bubble theorem (see \cite{aleksandrov1958uniqueness}).
\end{proof}
\section{Proof of Theorem~\ref{mainthm III}}\label{section pf mainthm III}

In this section, we are going to show Theorem~\ref{mainthm III}. The proof will follow that of Theorem~\ref{mainthm I}, with some adequate modifications. 

\begin{proof}[Proof of Theorem~\ref{mainthm III}]
Let $u$ be the solution to \eqref{heat} and assume that $u$ satisfies \eqref{interior odc} with respect to sequences $(t_n)_n$ and $(\tau_n)_n$ with accumulation points $t^\star,\tau^\star\in (0,\infty]$ respectively. By the unique solvability of \eqref{heat}, Corollary~\ref{cor termwise whence u solves} implies that
\begin{equation}\label{expression for the solution u}
    u(t):=\sum_{k=1}^\infty e^{-\Lambda_k t} \Phi_k, 
\end{equation}
where $\Phi_k$ are defined according to \eqref{eigenfunction expansion II}.

Take an arbitrary function $\psi\in H_*^{1/2}(\partial\omega)$. Then, \eqref{interior odc} yields:
\begin{equation*}
\begin{aligned}
0=a_n \int_{\partial\omega} \psi = \left(\restr{u(\tau_n)}{\partial\omega},\psi\right)_{L^2(\partial\om)}=\sum_{k=1}^\infty \left(\restr{\Phi_k}{\partial\omega},\psi \right)_{L^2(\partial\omega)} e^{-\Lambda_k \tau_n},\\
0=b_n \int_{\partial\omega} \psi = 
 \Pairing{\partial_\nu \restr{u(t_n)}{\partial\omega}}{\psi}=
\sum_{k=1}^\infty \Pairing{\partial_\nu \restr{\Phi_k}{\partial\omega}}{\psi} e^{-\Lambda_k t_n}.
\end{aligned}
\end{equation*}
Thus, along the same line as proof of Theorem~\ref{mainthm I}, in each case (recall that, by assumption, we have either $t^\star\in (0,\infty)$ or $t^\star=\infty$, and either $\tau^\star\in (0,\infty)$ or $\tau^\star=\infty$), one can show inductively that for all $k\in \mathbb{N}$ and $\psi\in H_*^{1/2}(\partial\om)$ the following holds:
\begin{equation}\label{overdetereminations on partial omeghino}
    \left(\restr{\Phi_k}{\partial\omega},\psi\right)_{L^2(\partial\omega)}=0\quad \text{and} \quad \Pairing{\partial_\nu \restr{\Phi_k}{\partial\omega}}{\psi}=0.
\end{equation}
Let now $\Phi\in H_0^1(\Omega)$ denote the following auxiliary function:
\begin{equation}\label{serrin solution diy2}
    \Phi:=\sum_{k=1}^\infty \lambda_k^{-1}\alpha_k \phi_k=
    \sum_{k=1}^\infty \Lambda_k^{-1}\Phi_k.
\end{equation}
By construction, \eqref{serrin solution diy2} converges in the $D(\Delta_D)$-norm. As a result, term-wise application of the Laplacian shows that $\Phi$ satisfies the following boundary value problem:
\begin{equation}\label{non overdeteremined bvp}
    -\Delta\Phi = 1 \quad \text{in }\Omega, \quad \Phi=0\quad \text{on }\partial\Omega.
\end{equation}
Furthermore, by \eqref{overdetereminations on partial omeghino}, the following hold for all $\psi\in H_*^{1/2}(\partial\omega)$: 
\begin{eqnarray}\label{first one}
 \left(\restr{\Phi}{\partial\omega}, \psi\right)_{L^2(\omega)}= \sum_{k=1}^\infty \Lambda_k^{-1} \left(\restr{\Phi_k}{\partial\omega}, \psi \right)_{L^2(\omega)}=0,\\
 \label{second one}
\Pairing{\partial_\nu \restr{\Phi}{\partial\om}}{\psi} = \sum_{k=1}^\infty \Lambda_k^{-1} \Pairing{\partial_\nu \restr{\Phi_k}{\partial\omega}}{\psi}=0.   
\end{eqnarray}  
In particular, \eqref{first one} implies that $\restr{\Phi}{\partial\omega}\equiv a$ for some $a\in\mathbb{R}$. 
In turn, by combining \eqref{second one} with Lemma~\ref{lem weak conormal is weak neumann}, we obtain that the function $U\in H_0^1(\omega)$ given by $U:=\restr{\Phi}{\om}-a$ is a weak solution to the following overdetermined problem: 
\begin{equation*}
    -\Delta U=1\quad \text{in }\omega,\quad U=0\quad \text{on }\partial\omega, \quad \partial_\nu U=b\quad \text{on }\partial\omega
\end{equation*}
for some constant $b\in \mathbb{R}$. By \cite{figalli&zhang2025_serrin_rough_domains}, $\omega$ must be a ball and $U$ is radial in $\omega$.

In particular, also $\Phi$ is also radial inside $\omega$, being $\Phi=U+a$.
Moreover, notice that $\Phi$ is real analytic inside the whole $\Omega$, being a solution to \eqref{non overdeteremined bvp}. This implies that $\Phi$ is radial in $\Omega$. 
Finally, recall that $\Phi\in C^0(\overline{\Omega})$, since $\partial\Omega$ is made entirely of regular points for the Dirichlet Laplacian. Now, since $\Phi=0$ on $\partial\Omega$ because of the boundary condition, and $\Phi>0$ in $\Omega$ by the maximum principle, we conclude that $\partial\Omega$ is a sphere concentric with $\omega$. In other words, $\omega$ and $\Omega$ are concentric balls, as claimed.  
\end{proof}

\section{Possible extensions to complete Riemannian manifolds}\label{section manifolds}

In this section, we discuss which parts of the previous sections extend to the
Riemannian setting, and which conclusions have to be modified. Throughout,
$(M, g)$ is a complete Riemannian manifold of dimension $N\ge 2$, $\Omega\subset \overline{\Omega}\subset M$ is a connected
domain with $C^\infty$ boundary, and $u=u(x,t)$ denotes the solution of the
Dirichlet heat problem 
\begin{equation}\label{eq:heat-riem}
\begin{cases}
(\partial_t+\Delta_g)u=0 & \text{in } (0,\infty)\times \Omega,\\
u=1 & \text{on } \Omega\times\{0\},\\
u=0 & \text{on } (0,\infty)\times \partial\Omega.
\end{cases}
\end{equation}
where $\Delta_g$ denotes the Laplace--Beltrami operator on $(M, g)$.
We write $\nu$ for the \emph{inward} unit normal along $\partial\Omega$ (the
choice of the normal only affects the sign of $\partial_\nu u$, and therefore
is immaterial for all the rigidity statements below).

By standard spectral theory for the Dirichlet Laplace--Beltrami operator on
compact manifolds with boundary, there exists an orthonormal basis
$(\phi_j)_{j\geq 1}\subset C^\infty(\overline{\Omega})$ of $L^2(\Omega)$ and
a nondecreasing sequence of positive eigenvalues $\lambda_j\to\infty$ such
that
\[
\Delta_g\phi_j=\lambda_j\phi_j\quad\text{in }\Omega,\qquad \phi_j=0
\quad\text{on }\partial\Omega.
\]
If
\[
a_j:=\int_\Omega \phi_j,
\]
then
\begin{equation}\label{eq:spectral-u}
u(x,t)=\sum_{j=1}^\infty a_j e^{-\lambda_j t}\phi_j(x),
\end{equation}
and, for every $\tau>0$,
\begin{equation}\label{eq:spectral-du}
\partial_\nu u(t)=\sum_{j=1}^\infty a_j e^{-\lambda_j t}\partial_\nu\phi_j
\qquad\text{in }C^\infty(\partial\Omega)\text{ for }t\geq \tau,
\end{equation}
see \cite[Chapters VI--VII]{MR768584}.

Let
\[
\mathcal C(\partial\Omega):=\{ \text{constant functions on }\partial\Omega\}.
\]
The boundary function $\partial_\nu u(t)$ belongs to
$\mathcal C(\partial\Omega)$ if and only if
\[
\int_{\partial\Omega}\psi\,\partial_\nu u(t)=0
\qquad\text{for every }\psi\in C^\infty(\partial\Omega)
\text{ such that }\int_{\partial\Omega}\psi=0.
\]

\begin{proposition}\label{prop:riem-cfp}
Assume that there exists a sequence $(t_m)_m\subset (0,\infty)$ such that
\[
\partial_\nu u(t_m)\in \mathcal C(\partial\Omega)\qquad\text{for all }m\in\mathbb{N},
\]
and either
\[
t_m\to t_\ast\in (0,\infty),
\qquad\text{or}\qquad
t_m\to \infty.
\]
Then $\Omega$ has the \emph{constant flow property}, namely
\[
\partial_\nu u(t)\in \mathcal C(\partial\Omega)\qquad\text{for every fixed }t>0.
\]
\end{proposition}

\begin{proof}
Fix $\psi\in C^\infty(\partial\Omega)$ with zero average on $\partial\Omega$ and
set
\[
F_\psi(t):=\int_{\partial\Omega}\psi\,\partial_\nu u(t).
\]
By \eqref{eq:spectral-du},
\begin{equation}\label{eq:Fpsi-series}
F_\psi(t)=\sum_{j=1}^\infty b_j e^{-\lambda_j t},
\qquad
b_j:=a_j\int_{\partial\Omega}\psi\,\partial_\nu\phi_j,
\end{equation}
with absolute and uniform convergence on $[\tau,\infty)$ for every $\tau>0$.
Hence $F_\psi$ is real-analytic on $(0,\infty)$.

If $t_m\to t_\ast\in(0,\infty)$, then $F_\psi(t_m)=0$ for every $m\in\mathbb{N}$ by the
assumption $\partial_\nu u(t_m)\in\mathcal C(\partial\Omega)$. Since the
zeros accumulate at an interior point of $(0,\infty)$ and $F_\psi$ is
real-analytic, we conclude that $F_\psi\equiv 0$ on $(0,\infty)$.

Assume now that $t_m\to\infty$. Let
\[
0<\Lambda_1<\Lambda_2<\Lambda_3<\cdots
\]
be the distinct Dirichlet eigenvalues of $\Delta_g$ on $\Omega$, and rewrite
\eqref{eq:Fpsi-series} as
\[
F_\psi(t)=\sum_{\ell=1}^\infty \beta_\ell e^{-\Lambda_\ell t},
\qquad
\beta_\ell:=\sum_{\lambda_j=\Lambda_\ell} b_j.
\]
We claim that $\beta_\ell=0$ for every $\ell$. Since the series converges
absolutely on $[T,\infty)$ for every $T>0$, we can choose $T>0$ so that
$t_m\geq 2T$ for all $m$ large enough. Then
\[
0=e^{\Lambda_1 t_m}F_\psi(t_m)
=\beta_1+\sum_{\ell\geq 2}\beta_\ell e^{-(\Lambda_\ell-\Lambda_1)t_m}.
\]
Moreover,
\begin{align*}
\sum_{\ell\geq 2} |\beta_\ell| e^{-(\Lambda_\ell-\Lambda_1)t_m}
&=
\sum_{\ell\geq 2} |\beta_\ell| e^{-\Lambda_\ell T}
e^{-\Lambda_\ell(t_m-T)}e^{\Lambda_1 t_m}\\
&\leq
e^{-(\Lambda_2-\Lambda_1)t_m+\Lambda_2 T}
\sum_{\ell\geq 2} |\beta_\ell|e^{-\Lambda_\ell T}\xrightarrow[m\to\infty]{}0.
\end{align*}
Hence $\beta_1=0$. Subtracting the first term and iterating the same argument,
we obtain $\beta_\ell=0$ for every $\ell$, hence $F_\psi\equiv 0$ also in this
case.

Since $F_\psi(t)=0$ for every zero-average $\psi$ and every $t>0$, it follows
that $\partial_\nu u(t)$ is orthogonal to all zero-average smooth
functions on $\partial\Omega$, and therefore
$\partial_\nu u(t)\in\mathcal C(\partial\Omega)$ for every $t>0$.
\end{proof}

\begin{corollary}\label{cor:riem-analytic}
Assume in addition that $(M,g)$ is real-analytic. Under the assumptions of
Proposition~\ref{prop:riem-cfp}, the domain $\Omega$ is an isoparametric tube.
In particular, if $M=\mathbb R^N$ or $M=\mathbb H^N$, then $\Omega$ is a
geodesic ball.
\end{corollary}

\begin{proof}
By Proposition~\ref{prop:riem-cfp}, $\Omega$ has the constant flow property.
Savo proved that, on a compact real-analytic Riemannian manifold with smooth
boundary, the constant flow property is equivalent to being an isoparametric
tube \cite{MR3859464}. This gives the first claim.

If $M=\mathbb R^N$ or $M=\mathbb H^N$, then compact isoparametric hypersurfaces
are geodesic spheres; equivalently, bounded isoparametric tubes are geodesic
balls. This follows from the classification of isoparametric hypersurfaces in
real space forms; see, for instance, the survey \cite[Section~1]{MR2434936}.
\end{proof}

The conclusion of Corollary~\ref{cor:riem-analytic} is sharp: on general
real-analytic manifolds, one cannot replace ``isoparametric tube'' with
``geodesic ball,'' not even on the round sphere.

\begin{proposition}\label{prop:sphere-tubes}
Let $N\ge 3$, $2\le k\le N-1$, and $0<r<1$. Then, the set
\begin{equation*}
    \Omega:=\setbld{(x,y)\in \mathbb{R}^k\times \mathbb{R}^{N-k}}{|x|^2+|y|^2=1,\quad |y|<r}\subset\mathbb{S}^{N-1}
\end{equation*}
is a domain of $\mathbb{S}^{N-1}$ with the constant flow property that is not a geodesic ball.
\end{proposition}
\begin{proof}
By construction, $\Omega$ is an isoparametric tube around a totally geodesic sphere $\mathbb{S}^{k-1}$ on $\mathbb{S}^{N-1}$. Its boundary is given by
\begin{equation*}
    \partial\Omega=\setbld{(x,y)\in \mathbb{R}^k\times \mathbb{R}^{N-k}}{|x|=\sqrt{1-r^2}, \quad |y|=r} \simeq \mathbb{S}^{k-1}\times \mathbb{S}^{N-k-1}.
\end{equation*}
As $\mathbb{S}^{k-1}\times \mathbb{S}^{N-k-1}$ is not diffeomorphic to a sphere, this readily implies that $\Omega$ cannot be a geodesic ball. 
In what follows, we will show that, despite that, $\Omega$ exhibits enough symmetry to ensure the constant flow property. To this end, consider the group $\Gamma:= O(k)\times O(N-k)$,
acting (in the natural way) on the set $\Omega\subset\mathbb{R}^N=\mathbb{R}^k\times\mathbb{R}^{N-k}$ as a subgroup of the orthogonal group in dimension $N$. First, notice that both $\Omega$ and $\partial\Omega$ are $\Gamma$-invariant, and that the differential operators $\Delta_g$ and $\partial_\nu$ are $\Gamma$-equivariant. This,  combined with the fact that the initial and boundary conditions in \eqref{eq:heat-riem} are $\Gamma$-invariant, yields that for any solution $u=u(x,t)$ of \eqref{eq:heat-riem} and $\gamma\in \Gamma$, then $u_\gamma(x,t):=u(\gamma(x), t)$ also solves \eqref{eq:heat-riem}. In particular, by unique solvability, it follows that any solution to \eqref{eq:heat-riem} must be $\Gamma$-invariant in space for all $t>0$. Finally, since $\Gamma$ acts transitively on $\partial\Omega$, then $\partial_\nu u(t)$ must be a constant function on $\partial\Omega$ for all $t>0$.  
\end{proof}

Proposition~\ref{prop:sphere-tubes} shows that the conclusion ``$\Omega$ is a
ball'' cannot hold in a na\"{i}ve Riemannian extension of the finite-time or
large-time rigidity result, even when the overdetermined condition holds for
\emph{every} $t>0$.

\begin{corollary}\label{cor:riem-serrin}
Assume that the hypotheses of Proposition~\ref{prop:riem-cfp} hold with
$t_m\to\infty$. Then the function
\[
v(x):=\int_0^\infty u(x,t)
\]
is well defined, belongs to $C^\infty(\overline\Omega)$, and solves the
overdetermined torsion problem
\[
\begin{cases}
\Delta_g v = 1 & \text{in }\Omega,\\
v=0 & \text{on }\partial\Omega,\\
\partial_\nu v = \text{const} & \text{on }\partial\Omega.
\end{cases}
\]
In particular, $\Omega$ is a Serrin domain.
\end{corollary}

\begin{proof}
By Proposition~\ref{prop:riem-cfp}, $\Omega$ has the constant flow property, i.e.
\[
\partial_\nu u(t)=c(t)\qquad\text{on }\partial\Omega
\]
for some function $c:(0,\infty)\to\mathbb R$. By \eqref{eq:spectral-u}, the
solution $u$ decays exponentially in $C^\infty(\overline\Omega)$ as
$t\to\infty$, so the integral defining $v$ converges in
$C^\infty(\overline\Omega)$. Since $u=0$ on $\partial\Omega\times (0,\infty)$,
we have $v=0$ on $\partial\Omega$. Moreover
\[
\partial_\nu v
=
\int_0^\infty \partial_\nu u(t)
=
\int_0^\infty c(t),
\]
which is constant on $\partial\Omega$.

Finally, integrating the heat equation in time gives
\[
\Delta_g v
=
-\int_0^\infty \partial_t u(t)
=
u(0)-\lim_{t\to\infty}u(t)
=
1
\qquad\text{in }\Omega.
\]
\end{proof}

\section{Conjectures and open problems}
In this section, we state some open problems and conjectures related to the overdetermined problem presented in this paper.

\begin{problem}\label{problem unbounded}
For what unbounded domains $\Omega\subset\mathbb{R}^N$ ($N\ge 2$) does the solution to the following problem satisfy \eqref{boundary odc}?
\begin{equation*}
  \begin{cases}
      \partial_t u= \Delta u \quad \text{in } \Omega\times (0,\infty),\\
      u=1 \quad \text{on }\Omega\times\{0\},\\
      u=0\quad \text{on }\partial\Omega\times(0,\infty),\\
      u(x,t)\to 0 \quad \text{uniformly in $t$ as $|x|\to \infty$}.
  \end{cases}  
\end{equation*}
\end{problem}

\begin{conjecture}\label{conjecture exterior}
Complements of closed balls are the only solutions to Problem \ref{problem unbounded} among \emph{exterior domains}.      
\end{conjecture}

\begin{remark}
Conjecture \ref{conjecture exterior} is in line with the results of \cite{MR1860493}. There, the authors make use of the method of moving planes in a cylindrical region of $\mathbb{R}^N\times\mathbb{R}$, thus fully exploiting a continuous-time overdetermination. As a result, it is not clear how to adapt their methods to the case of a discrete-time overdetermination, as in \eqref{boundary odc}. We also remark that the eigenfunction expansion technique employed in the proof of Theorem \ref{mainthm I} cannot be applied in the unbounded domain case due to the lack of compactness.
\end{remark}

\begin{problem}
Generalize Theorem \ref{mainthm II} to milder regularity conditions.   
\end{problem}
\begin{remark}
Our current methods rely on the following two ingredients: precise heat content asymptotics and Aleksandrov's soap bubble theorem. While the latter is known to hold (in a geometric-measure-theoretical sense) for sets of finite perimeter, the heat content asymptotics provided in \cite{gilkey2003asymptotic} (as well as the analogous point-wise asymptotics given in \cite{savo2016heat}) 
used in the proof of Theorem \ref{mainthm II} rely on an infinite family of recurrence relations, which can only be derived when the boundary $\partial\Omega$ is smooth.
\end{remark}

\begin{problem}\label{problem finite times}
 For what bounded domains $\Omega\subset\mathbb{R}^N$ ($N\ge 2$) does the solution $u$ to \eqref{heat} admit \emph{a finite number} of times $t_1,\dots,t_m>0$ satisfying the following? 
 \begin{equation*}
     \partial_\nu u(t_n)\equiv b_n \quad \text{on }\partial\Omega\quad \text{for }n=1,\dots, m.
 \end{equation*}
\end{problem}

\begin{conjecture}
For any $m\in\mathbb{N}$, Problem \ref{problem finite times} admits nontrivial solutions (solutions that are not Euclidean balls). 
\end{conjecture}

\section*{Acknowledgements}
Lorenzo Cavallina was partially supported by JSPS KAKENHI Grant Numbers JP22K13935, JP26K17010, and JP21KK0044.\\
Andrea Pinamonti is a member of the Istituto Nazionale di Alta Matematica (INdAM), Gruppo Nazionale per l'Analisi Matematica, la Probabilit\`a e le loro Applicazioni (GNAMPA), and is supported by the University of Trento, the MIUR-PRIN 2022 Project \emph{Regularity problems in sub-Riemannian structures}  Project code: 2022F4F2LH and the INdAM-GNAMPA 2025 Project \emph{Structure of sub-Riemannian hypersurfaces in Heisenberg groups}, CUP ES324001950001.

\bibliographystyle{siam}
\bibliography{references}

@article {MR1011556,
    AUTHOR = {Alessandrini, Giovanni and Garofalo, Nicola},
     TITLE = {Symmetry for degenerate parabolic equations},
   JOURNAL = {Arch. Rational Mech. Anal.},
  FJOURNAL = {Archive for Rational Mechanics and Analysis},
    VOLUME = {108},
      YEAR = {1989},
    NUMBER = {2},
     PAGES = {161--174},
      ISSN = {0003-9527},
   MRCLASS = {35K55 (35K65)},
  MRNUMBER = {1011556},
MRREVIEWER = {Zhuo\ Qun\ Wu},
       DOI = {10.1007/BF01053461},
       URL = {https://doi.org/10.1007/BF01053461},
}

@Misc{siam,
  key = {zzz},
  title =	 {{SIAM} Style Manual: For journals and books},
  year =	 2013,
url = {https://www.siam.org/journals/pdf/stylemanual.pdf}}

@article{li_yau1983schrodinger&eigenvalue_problem,
  title={On the {S}chr{\"o}dinger equation and the eigenvalue problem},
  author={Li, Peter and Yau, Shing-Tung},
  journal={Communications in Mathematical Physics},
  volume={88},
  number={3},
  pages={309--318},
  year={1983},
  publisher={Springer}
}

@article{figalli&zhang2025_serrin_rough_domains,
    author ={Figalli, Alessio and  Zhang, Yi Ru-Ya},    
    title = {{S}errin’s overdetermined problem in rough domains} ,
    journal = {Journal of the European Mathematical Society (published online first)},
    year = {2025},
    doi={DOI 10.4171/JEMS/1726}
}

@article {MR1776351,
    AUTHOR = {Damascelli, Lucio and Pacella, Filomena},
     TITLE = {Monotonicity and symmetry results for {$p$}-{L}aplace
              equations and applications},
   JOURNAL = {Adv. Differential Equations},
  FJOURNAL = {Advances in Differential Equations},
    VOLUME = {5},
      YEAR = {2000},
    NUMBER = {7-9},
     PAGES = {1179--1200},
      ISSN = {1079-9389},
   MRCLASS = {35J65 (35B05 35B50 35J70)},
  MRNUMBER = {1776351},
MRREVIEWER = {Srinivasan\ Kesavan},
}

@article{kumaresan_prajapat1998serrin,
  title={Serrin’s result for hyperbolic space and sphere},
  author={Kumaresan, Somas and Prajapat, Jyotshana},
  journal={Duke Mathematical Journal},
  volume={91},
  number={1},
  pages={17--28},
  year={1998}
}

@book{gilkey2003asymptotic,
  title={Asymptotic formulae in spectral geometry},
  author={Gilkey, Peter B},
  year={2003},
  publisher={Chapman and Hall/CRC}
}

@article{vandenberg&Gilkey&Seeley2008heat_content_asymptotics,
  title={Heat content asymptotics with singular initial temperature distributions},
  author={Van den Berg, M and Gilkey, P and Seeley, R},
  journal={Journal of Functional Analysis},
  volume={254},
  number={12},
  pages={3093--3122},
  year={2008},
  publisher={Elsevier}
}

@article {MR2231874,
    AUTHOR = {Magnanini, R. and Prajapat, J. and Sakaguchi, S.},
     TITLE = {Stationary isothermic surfaces and uniformly dense domains},
   JOURNAL = {Trans. Amer. Math. Soc.},
  FJOURNAL = {Transactions of the American Mathematical Society},
    VOLUME = {358},
      YEAR = {2006},
    NUMBER = {11},
     PAGES = {4821--4841},
      ISSN = {0002-9947,1088-6850},
   MRCLASS = {58E12 (35K05 35K15 53A10)},
  MRNUMBER = {2231874},
MRREVIEWER = {Sergey\ G.\ Pyatkov},
       DOI = {10.1090/S0002-9947-06-04145-6},
       URL = {https://doi.org/10.1090/S0002-9947-06-04145-6},
}

@article{vogel1992symmetry_and_regularity,
    author = {A. L. Vogel},
    title =  {Symmetry and regularity for general regions having a solution to certain overdetermined boundary value problems},
    journal = {Atti Sem. Mat. Fis. Univ. Modena},
    year = 1992,
    volue=40,  
    number=2,
    pages={443--484}
}

@article{jerison_kenig1995inhomogeneous,
  title={The inhomogeneous {D}irichlet problem in {L}ipschitz domains},
  author={Jerison, David and Kenig, Carlos E},
  journal={Journal of Functional Analysis},
  volume={130},
  number={1},
  pages={161--219},
  year={1995},
  publisher={Academic Press}
}

@Misc{amsmath,
  author =	 {{American Mathematical Society}},
  title =	 {User's Guide for the \texttt{amsmath} Package
                  (Version 2.0)},
  url =		 {ftp://ftp.ams.org/pub/tex/doc/amsmath/amsldoc.pdf},
  urldate =	 {2015-07-30},
  year =	 2002}

@article{aleksandrov1958uniqueness,
  title={Uniqueness theorems for surfaces in the large {V}},
  author={Aleksandrov, Aleksandr D},
  journal={Amer. Math. Soc. Transl.(2)},
  volume={21},
  pages={412--416},
  year={1962},
 note={translated from Vestnik Leningrad Univ. 19:13 (1958) 5--8}
}

@article{serrin1971,
    author = {James Serrin},
    title = {A symmetry problem in potential theory},  
    doi={https://doi.org/10.1007/BF00250468},
    journal = {Archive for Rational Mechanics and Analysis},
    volume={43},
    year = {1971},
    pages={304--318}
}

@book{mclean2000strongly,
  title={Strongly elliptic systems and boundary integral equations},
  author={McLean, William Charles Hector},
  year={2000},
  publisher={Cambridge University Press}
}

@article{sakaguchibessatsu,
    author = {Sakaguchi, Shigeru},
    title = {Two-phase heat conductors with a
stationary isothermic surface and their
related elliptic overdetermined problems
},
    journal = {RIMS K{\^o}ky{\^u}roku Bessatsu},
    year = {2020},
    volume={B80},
    pages={113--132}
}

@article {MR3859464,
    AUTHOR = {Savo, Alessandro},
     TITLE = {Geometric rigidity of constant heat flow},
   JOURNAL = {Calc. Var. Partial Differential Equations},
  FJOURNAL = {Calculus of Variations and Partial Differential Equations},
    VOLUME = {57},
      YEAR = {2018},
    NUMBER = {6},
     PAGES = {Paper No. 156, 26},
      ISSN = {0944-2669,1432-0835},
   MRCLASS = {58J35 (35P15 35R01)},
  MRNUMBER = {3859464},
MRREVIEWER = {Qihua\ Ruan},
       DOI = {10.1007/s00526-018-1434-7},
       URL = {https://doi.org/10.1007/s00526-018-1434-7},
}

@book {MR768584,
    AUTHOR = {Chavel, Isaac},
     TITLE = {Eigenvalues in {R}iemannian geometry},
    SERIES = {Pure and Applied Mathematics},
    VOLUME = {115},
      NOTE = {Including a chapter by Burton Randol,
              With an appendix by Jozef Dodziuk},
 PUBLISHER = {Academic Press, Inc., Orlando, FL},
      YEAR = {1984},
     PAGES = {xiv+362},
      ISBN = {0-12-170640-0},
   MRCLASS = {58G25 (35P99 53C20)},
  MRNUMBER = {768584},
MRREVIEWER = {G\'erard\ Besson},
}

@article {MR2434936,
    AUTHOR = {Cecil, Thomas E.},
     TITLE = {Isoparametric and {D}upin hypersurfaces},
   JOURNAL = {SIGMA Symmetry Integrability Geom. Methods Appl.},
  FJOURNAL = {SIGMA. Symmetry, Integrability and Geometry. Methods and
              Applications},
    VOLUME = {4},
      YEAR = {2008},
     PAGES = {Paper 062, 28},
      ISSN = {1815-0659},
   MRCLASS = {53C40},
  MRNUMBER = {2434936},
MRREVIEWER = {Ronaldo\ Alves\ Garcia},
       DOI = {10.3842/SIGMA.2008.062},
       URL = {https://doi.org/10.3842/SIGMA.2008.062},
}

@article{Sakaguchi_hyperplanes2020,
title = {Some characterizations of parallel hyperplanes in multi-layered heat conductors},
journal = {Journal de Math\'ematiques Pures et Appliqu\'ees},
volume = {140},
pages = {185--210},
year = {2020},
issn = {0021-7824},
doi = {https://doi.org/10.1016/j.matpur.2020.06.007},
url = {https://www.sciencedirect.com/science/article/pii/S0021782420301057},
author = {Shigeru Sakaguchi},
}

@article{BisSavo,
    author = {Bisterzo, Andrea and Savo, Alessandro},
    title = {Rigidity of an overdetermined heat equation and minimal helicoids in space-forms},
    journal ={arXiv preprint arXiv:2507.08389},
    year=2025
}

@article{cavallina_poggesi2025elliptic&parabolic,
  title={Elliptic and parabolic overdetermined problems in multi-phase settings},
  author={Cavallina, Lorenzo and Poggesi, Giorgio},
  journal={arXiv preprint arXiv:2504.18808},
  year={2025}
}

@article{weinberger1971remark,
  title={Remark on the preceding paper of {S}errin},
  author={Weinberger, Hans F},
  journal={Archive for Rational Mechanics and Analysis},
  volume={43},
  number={4},
  pages={319--320},
  year={1971},
  publisher={Springer}
}

@article{fall&minled&weth2018serrin_on_the_sphere,
  title={Serrin’s overdetermined problem on the sphere},
  author={Fall, Mouhamed Moustapha and Minlend, Ignace Aristide and Weth, Tobias},
  journal={Calculus of Variations and Partial Differential Equations},
  volume={57},
  number={1},
  pages={3},
  year={2018},
  publisher={Springer}
}

@article{prajapat1998serrin,
  title={{S}errin’s result for domains with a corner or cusp},
  author={Prajapat, Jyotshana},
  journal={Duke Math. J.},
  volume={95},
  number={1},
  pages={29--31},
  year={1998}
}

@book{GilbargTrudinger,
    author = {Gilbarg, David and Trudinger, Neil S.},
    title = {Elliptic Partial Differential Equations of Second Order},
    publisher = {Springer Berlin, Heidelberg},
    series={Classics in Mathematics},
    year = 2001,
    edition= {2nd}
}

@article{savo2016heat,
  title={Heat flow, heat content and the isoparametric property},
  author={Savo, Alessandro},
  journal={Mathematische Annalen},
  volume={366},
  number={3},
  pages={1089--1136},
  year={2016},
  publisher={Springer}
}

@article {MR1860493,
    AUTHOR = {Garofalo, Nicola and Sartori, Elena},
     TITLE = {Symmetry in a free boundary problem for degenerate parabolic
              equations on unbounded domains},
   JOURNAL = {Proc. Amer. Math. Soc.},
  FJOURNAL = {Proceedings of the American Mathematical Society},
    VOLUME = {129},
      YEAR = {2001},
    NUMBER = {12},
     PAGES = {3603--3610},
      ISSN = {0002-9939,1088-6826},
   MRCLASS = {35R35 (35K20 35K55 35K65)},
  MRNUMBER = {1860493},
MRREVIEWER = {Domingo\ A.\ Tarzia},
       DOI = {10.1090/S0002-9939-01-05993-7},
       URL = {https://doi.org/10.1090/S0002-9939-01-05993-7},
}

@article {Garofalo,
    AUTHOR = {Garofalo, Nicola and Lewis, John L.},
     TITLE = {A symmetry result related to some overdetermined boundary
              value problems},
   JOURNAL = {Amer. J. Math.},
  FJOURNAL = {American Journal of Mathematics},
    VOLUME = {111},
      YEAR = {1989},
    NUMBER = {1},
     PAGES = {9--33},
      ISSN = {0002-9327,1080-6377},
   MRCLASS = {35N10 (35J60)},
  MRNUMBER = {980297},
MRREVIEWER = {V.\ S.\ Rabinovich},
       DOI = {10.2307/2374477},
       URL = {https://doi.org/10.2307/2374477},
}

@article {MR3955113,
    AUTHOR = {Fogagnolo, Mattia and Mazzieri, Lorenzo and Pinamonti, Andrea},
     TITLE = {Geometric aspects of {$p$}-capacitary potentials},
   JOURNAL = {Ann. Inst. H. Poincar\'e{} C Anal. Non Lin\'eaire},
  FJOURNAL = {Annales de l'Institut Henri Poincar\'e{} C. Analyse Non
              Lin\'eaire},
    VOLUME = {36},
      YEAR = {2019},
    NUMBER = {4},
     PAGES = {1151--1179},
      ISSN = {0294-1449,1873-1430},
   MRCLASS = {35J92 (31C45)},
  MRNUMBER = {3955113},
MRREVIEWER = {Juha\ K.\ Kinnunen},
       DOI = {10.1016/j.anihpc.2018.11.005},
       URL = {https://doi.org/10.1016/j.anihpc.2018.11.005},
}

@incollection {MR3893584,
    AUTHOR = {Magnanini, Rolando},
     TITLE = {Alexandrov, {S}errin, {W}einberger, {R}eilly: simmetry and
              stability by integral identities},
 BOOKTITLE = {Bruno {P}ini {M}athematical {A}nalysis {S}eminar 2017},
    SERIES = {Bruno Pini Math. Anal. Semin.},
    VOLUME = {8},
     PAGES = {121--141},
 PUBLISHER = {Univ. Bologna, Alma Mater Stud., Bologna},
      YEAR = {2017},
   MRCLASS = {35N25 (35-02 35A23 35B35 53A10)},
  MRNUMBER = {3893584},
}

@article {MR3802818,
    AUTHOR = {Nitsch, C. and Trombetti, C.},
     TITLE = {The classical overdetermined {S}errin problem},
   JOURNAL = {Complex Var. Elliptic Equ.},
  FJOURNAL = {Complex Variables and Elliptic Equations. An International
              Journal},
    VOLUME = {63},
      YEAR = {2018},
    NUMBER = {7-8},
     PAGES = {1107--1122},
      ISSN = {1747-6933,1747-6941},
   MRCLASS = {35N25 (26D15 35B50 35J05)},
  MRNUMBER = {3802818},
       DOI = {10.1080/17476933.2017.1410798},
       URL = {https://doi.org/10.1080/17476933.2017.1410798},
}

@article {MR2448319,
    AUTHOR = {Brandolini, B. and Nitsch, C. and Salani, P. and Trombetti,
              C.},
     TITLE = {Serrin-type overdetermined problems: an alternative proof},
   JOURNAL = {Arch. Ration. Mech. Anal.},
  FJOURNAL = {Archive for Rational Mechanics and Analysis},
    VOLUME = {190},
      YEAR = {2008},
    NUMBER = {2},
     PAGES = {267--280},
      ISSN = {0003-9527,1432-0673},
   MRCLASS = {35N10 (35B06 35J60)},
  MRNUMBER = {2448319},
MRREVIEWER = {Maria\ Santos\ Bruz\'on},
       DOI = {10.1007/s00205-008-0119-3},
       URL = {https://doi.org/10.1007/s00205-008-0119-3},
}

@article {MR2545870,
    AUTHOR = {Cianchi, Andrea and Salani, Paolo},
     TITLE = {Overdetermined anisotropic elliptic problems},
   JOURNAL = {Math. Ann.},
  FJOURNAL = {Mathematische Annalen},
    VOLUME = {345},
      YEAR = {2009},
    NUMBER = {4},
     PAGES = {859--881},
      ISSN = {0025-5831,1432-1807},
   MRCLASS = {35N10 (35J35)},
  MRNUMBER = {2545870},
       DOI = {10.1007/s00208-009-0386-9},
       URL = {https://doi.org/10.1007/s00208-009-0386-9},
}

@article {MR2232009,
    AUTHOR = {Fragal\`a, Ilaria and Gazzola, Filippo and Kawohl, Bernd},
     TITLE = {Overdetermined problems with possibly degenerate ellipticity,
              a geometric approach},
   JOURNAL = {Math. Z.},
  FJOURNAL = {Mathematische Zeitschrift},
    VOLUME = {254},
      YEAR = {2006},
    NUMBER = {1},
     PAGES = {117--132},
      ISSN = {0025-5874,1432-1823},
   MRCLASS = {35J60 (35B50 35J25 35J70 35N10)},
  MRNUMBER = {2232009},
MRREVIEWER = {Luisa\ Moschini},
       DOI = {10.1007/s00209-006-0937-7},
       URL = {https://doi.org/10.1007/s00209-006-0937-7},
}

@article {MR2366129,
    AUTHOR = {Farina, A. and Kawohl, B.},
     TITLE = {Remarks on an overdetermined boundary value problem},
   JOURNAL = {Calc. Var. Partial Differential Equations},
  FJOURNAL = {Calculus of Variations and Partial Differential Equations},
    VOLUME = {31},
      YEAR = {2008},
    NUMBER = {3},
     PAGES = {351--357},
      ISSN = {0944-2669,1432-0835},
   MRCLASS = {35N10 (35J65)},
  MRNUMBER = {2366129},
       DOI = {10.1007/s00526-007-0115-8},
       URL = {https://doi.org/10.1007/s00526-007-0115-8},
}

@article {MR2375699,
    AUTHOR = {Magnanini, Rolando and Sakaguchi, Shigeru},
     TITLE = {Stationary isothermic surfaces for unbounded domains},
   JOURNAL = {Indiana Univ. Math. J.},
  FJOURNAL = {Indiana University Mathematics Journal},
    VOLUME = {56},
      YEAR = {2007},
    NUMBER = {6},
     PAGES = {2723--2738},
      ISSN = {0022-2518,1943-5258},
   MRCLASS = {35K05 (35B05 35K20)},
  MRNUMBER = {2375699},
MRREVIEWER = {Neil\ A.\ Watson},
       DOI = {10.1512/iumj.2007.56.3150},
       URL = {https://doi.org/10.1512/iumj.2007.56.3150},
}

@article{CavallinaMagnaniniSakaguchi2021constantflowproperty,
  title={Two-phase heat conductors with a surface of the constant flow property},
  author={Cavallina, Lorenzo and Magnanini, Rolando and Sakaguchi, Shigeru},
  journal={The Journal of Geometric Analysis},
  volume={31},
  number={1},
  pages={312--345},
  year={2021},
  publisher={Springer}
}

@article{sakaguchi2016rendiconti,
    author = {Sakaguchi, Shigeru},
    title = {Two-phase heat conductors with
a stationary isothermic surface},
    journal = {Rend. Istit. Mat. Univ. Trieste},
    year = 2016,
    volume=48,
    pages={167--187}
}
\bigskip

\noindent
\textsc{Lorenzo Cavallina:}\\
\noindent
\textsc{
Mathematical Institute, Tohoku University, Aoba-ku, 
Sendai 980-8578, Japan}\\
\noindent
{\em Electronic mail address:}
cavallina.lorenzo.e6@tohoku.ac.jp

\bigskip

\noindent
\textsc{Andrea Pinamonti:}\\
\noindent
\textsc{Dipartimento di Matematica, Universit\`a di Trento,
Via Sommarive, 14, 38123 Povo TN, Italy
}
\\
\noindent
{\em Electronic mail address:}
andrea.pinamonti@unitn.it

\end{document}